\documentclass[a4paper,11pt,reqno]{article}

\usepackage{a4wide, fullpage}
\usepackage[latin1]{inputenc}

\usepackage{graphicx}
\usepackage{amsmath}
\usepackage{amsthm}
\usepackage{amssymb}
\usepackage{color}
\usepackage[normalem]{ulem} 
\usepackage {ulem} 

\usepackage{enumerate}

\theoremstyle{plain}
\newtheorem{theorem}{Theorem}[section]
\newtheorem{proposition}[theorem]{Proposition}
\newtheorem{lemma}[theorem]{Lemma}

\newtheorem{definition}[theorem]{Definition}

\theoremstyle{definition}

\theoremstyle{remark}
\newtheorem{remark}[theorem]{Remark}

\newcommand{\N}{\ensuremath{\mathbb{N}}}
\newcommand{\Z}{\ensuremath{\mathbb{Z}}}
\newcommand{\R}{\ensuremath{\mathbb{R}}}

\newcommand{\E}{\ensuremath{\mathbb{E}}}

\newcommand{\var}{\ensuremath{{\rm{var}}}}
\newcommand{\cov}{\ensuremath{{\rm{cov}}}}

\newcommand{\noemi}[1]{\textcolor{black}{\textrm{#1}}}

\title{The ancestral process of long-range seed bank models}
\author{Jochen Blath\footnote{TU Berlin, Institut f\"ur Mathematik, Strasse des 17. Juni 136, D-10623 Berlin. \emph{blath@math.tu-berlin.de}},
 Adri\'an Gonz\'alez Casanova\footnote{Berlin Mathematical School and RTG 1845, TU Berlin, Strasse des 17. Juni 136, D-10623 Berlin. \emph{adriangcs@hotmail.com}},
 Noemi Kurt\footnote{TU Berlin, Institut f\"ur Mathematik, Strasse des 17. Juni 136, D-10623 Berlin. \emph{kurt@math.tu-berlin.de}}, 
Dario Span\`o\footnote{Department of Statistics, University of Warwick, Coventry CV4 7AL, UK. \emph{D.Spano@warwick.ac.uk}. D.S.'s research is partly supported by  CRiSM, an EPSRC-HEFCE UK grant.}}
\date{\today}

\begin{document}
\maketitle

\begin{abstract}
We present a new model for seed banks, where \noemi{direct ancestors of individuals may have lived in the near as well as the very far past.}
The classical Wright-Fisher model, 
as well as a seed bank model with bounded age distribution considered by Kaj, Krone and Lascoux (2001)
are special cases of our model. We discern three
parameter regimes of the seed bank age distribution, which lead to substantially different behaviour in terms of genetic variability, 
in particular with respect to fixation of types and  time to the most recent common ancestor. We prove that for age 
distributions with finite mean, the ancestral process converges to a time-changed Kingman coalescent, while in the case of infinite mean, ancestral lineages might not merge at all with positive probability. Further, we present
a construction of the forward in time process in equilibrium. The mathematical methods are based on renewal theory, the urn process introduced by Kaj et al.,
 as well as on a paper
by Hammond and Sheffield (2011).\\

\end{abstract}

\section{Introduction}

In this paper we discuss a new mathematical model for the description of the genetic variability of neutral haploid populations of fixed size under the influence of a 
general {\em seed bank} effect. In contrast to previous models, such as the Kaj, Krone and Lascoux model \cite{KKL}, we are particularly interested in
situations where ancestors of individuals of the present generation may have lived in the rather remote past.

Seed banks are of significant evolutionary importance, and come in various guises. Typical situations range from plant seeds which fall dormant for several 
generations during unfavourable ecological circumstances \cite{Tellier, Vitalis}, fruit tissue preserved in Siberian permafrost
 \cite{Flower}, to bacteria turning into 
{\em endospores} if the concentration of nutrients in the environment falls 
below a certain threshold. Such endospores may in principle persist for an unlimited amount of time before they become active again (see, e.g.\ \cite{Cano}).
Seed bank related effects can be viewed as sources of genetic novelty \cite{Levin1990} and are generally believed to increase
 observed genetic variability. \\
In \cite{KKL}, a mathematical model for a (weak) seedbank effect is investigated, with the number of generations backwards in time that may influence the current population being bounded by a constant $m$ and being small when compared to the total population size (resp.\ during passage to a scaling limit).
Under such circumstances, it is then shown that the ancestral process of the population can be approximately described by a time-changed Kingman coalescent,
where the (constant) time change leads to a linear decrease of the coalescence rates of ancestral lineages depending on the square of the expected seedbank age distribution. 
Overall, genetic variability is thus increased (in particular if mutation is taken into account), but the qualitative features of the ancestral history of the population remain unchanged.

In the present paper, we consider \noemi{the ancestral process of } a neutral seed bank model with Wright-Fisher-type dynamics, assuming constant population size $N$. However, the distance measured in generations between direct ancestor and potential offspring will not assumed to be bounded, but rather sampled according to some (potentially unbounded) age distribution $\mu$ on $\mathbb N$. 
For $\mu=\delta_1$, we \noemi{recover the ancestral process of} the classical Wright-Fisher model, and scaling by the population size yields a Kingman coalescent as limiting ancestral process. For $\mu$ with bounded support, say with a maximum value $m$, independent of $N,$
we are in the setup of \cite{KKL}, and obtain a time change of Kingman's coalescent appearing in the limit 
(again after classical scaling).

Yet, some species suggest (i.e.\ bacteria transforming into endospores) that $\mu$ could be effectively unbounded, in particular non-negligible when compared to the population size. This can lead to entirely different regimes.

Our first result is that if $\mu$ has finite expectation, we again obtain a time-changed Kingman's coalescent after classical
rescaling. The behaviour of the model however changes completely if we assume $\mu$ to have infinite expectation. A natural example for age-distributions is a discrete measure $\mu$ with a power-law decay, that is
$$\mu(\{n,n+1,...\})=n^{-\alpha}L(n)$$
for some $\alpha >0$ and some slowly varying function $L.$ Depending on the choice of $\alpha$, we investigate the 
time to the most recent common ancestor (MRCA) of two individuals, if it exists. It turns out (Theorem \ref{thm:tmrca}) that for $\alpha>1/2,$ there is always
a common ancestor, but the expected time to the MRCA is finite if $\alpha>1$ and infinite if $\alpha<1.$ If $\alpha<1/2,$ any two 
ancestral lineages never meet at all with positive probability. 

\noemi{In the following section, we construct our model and present the main results. The proofs are given in Section 3. }

\section{Model and main results}

We work in discrete time (measured in units of non-overlapping generations) and with fixed finite population size $N\in\N.$ Time
 in generations is indexed by $\Z.$ The dynamics of the population forwards in time is given in the following way: 
Each individual chooses the generation of its father according to a law $\mu$
on $\N,$ \noemi{meaning that $\mu(n)$ gives the probability that the immediate ancestor of an individual of generation $i$ has lived in generation $i-n.$ We call $\mu$ the seedbank age distribution.} To avoid technicalities, we will always assume $\mu(\{1\})>0.$ After having chosen the generation, the individual
\noemi{picks the father unifomly} among the $N$ possible ancestors from that generation.

For concreteness, we will often assume that the age distribution $\mu$ is of the form $\mu=\mu_\alpha,$ with 
$$\mu_\alpha(\{n,n+1,...\})=n^{-\alpha}L(n), \quad n\in\N,$$
for some $\alpha \in (0,\infty)$ and some slowly varying function $L.$ Let $\Gamma_\alpha:=\{\mu_\alpha\}, \alpha\in(0,\infty)$
denote the set of all measures 
$\mu$ of this form. 
We are interested in the question of whether or not \noemi{in such a population a genetic type eventually fixates,} and if this happens in finite time almost surely. In the backward picture, this is related to asking if a finite set of 
individuals has a most recent common ancestor and when it lived.

\noemi{It turns out that in the above construction an ancestral line} can be described by a renewal process with interarrival law 
$\mu.$ The question of \noemi{existence of a common ancestor} and time to the most recent common ancestor can therefore be investigated via classical 
results of Lindvall \cite{Lindvall} on coupling times of discrete renewal processes, which are controlled in the 
power law case via applications of Karamata's Tauberian Theorem for power series, see e.g.\ \cite{BGT}. \noemi{This leads to three different regimes, see Theorem \ref{thm:tmrca}. If on the other hand one is interested in the forward in time process, mathematical modelling problems arise: In order to obtain a new generation of such a population, one requires information about the whole history, i.e. needs to start sampling at \lq$-\infty$\rq. In subsection \ref{sect:forward} we present a construction of such a population {\em in equilibrium}, which allows to study the correlations of the allele frequency process. This construction can be formalized in terms of Gibbs measures, following a paper of Hammond and Sheffield 
\cite{HammondSheffield}, where the case $N=1$ is considered in order to construct a discrete process with long-range 
correlations that converges to fractional
Brownian motion. This is sketched in the appendix.}

\subsection{Renewal construction of ancestral lineages and time to the most recent common ancestor}

We start with a descripton of the ancestral lineages of samples in our model in terms of renewal theory. Fix $N\in\N$ and a 
probability measure $\mu$ on the natural numbers. Let $v\in V_N:=\Z\times\{1,...,N\}$ denote an individual of our population. For $v\in V_N$ we write $v=(i_v,k_v)$ with $i_v\in\Z,$ and $1\leq k_v\leq N,$ hence $i_v$ indicating the generation 
of the individual in $\mathbb Z$, and $k_v$ the label among the $N$ individuals alive in this generation. 

The ancestral line $A(v)=\{v_0=v, v_1, v_2, \dots \}$ of our individual $v$ is a set of sites in $V_N$,
where $i_{v_0}, i_{v_1}, \dots \downarrow - \infty$ is a strictly decreasing sequence of generations,  with independent
decrements $i_{v_l}-i_{v_{l-1}}=:\eta_l, l\geq 1$ with distribution $\mu$, and where the $k_{v_0}, k_{v_1}, \dots$
are i.i.d.\ Laplace random variables with values in $\{1, \dots, N\}$, independent of $\{i_{v_l}\}_{l\in\N_0}$. Letting 
\noemi{$$S_n:=\sum_{l=0}^n \eta_l,$$} where we assume \noemi{$S_0=\eta_0=0$, we obtain a discrete renewal process with interarrival law $\mu.$ In the language of 
\cite{Lindvall_Book}, we say that a renewal takes place at each of the times $S_n,n\geq 0,$} and we write 
$(q_n)_{n\in\N_0}$ for the renewal sequence, that is, $q_n$ is the probability that $n$ is a renewal time. \\

It is now straightfoward to give a formal construction of the full ancestral process starting from $N$ individuals at time $0$
in terms of a family of $N$ independent renewal processes with interarrival law $\mu$ and a sequence of independent
uniform random variables $U^r(i), i \in -\mathbb N, r \in \{1, \dots, N\}$, with values in $\{1, \dots, N\}$ (independent also of the renewal processes).
Indeed, let the ancestral processes pick previous generations according to their respective renewal times, and then among
 the generations pick labels 
according to their respective
uniform random variables. As soon as at least two ancestral lineages hit a joint ancestor, their renewal processes couple, i.e.\ follow the 
same realization of one of their driving renewal processes (chosen arbitrarily, and discarding those remaining parts of the renewal processes and renewal times 
which aren't needed anymore).  In other words, their ancestral lines merge.

Denote by $P_N^\mu$ the law of the above ancestral process. For $v\in V_N$ with $i_v=0$, we have
%
%
\begin{equation}\label{eq:qn}q_n=P_N^\mu\Big(A(v)\cap\big(\{-n\}\times\{1,...,N\}\big)\neq \emptyset\Big),\end{equation}
and the probability
that $w\in V_N$ is an ancestor of $v,$ for $i_w<i_v,$ is given by 
$$P_N^\mu(w\in A(v))=\frac{1}{N}q_{i_v-i_w}.$$
For notational convenience, let us extend $q_n$ to
$n\in\Z$ by setting $q_n=0$ if $n<0.$ Note that $q_0=1.$

In \cite{KKL} it was proved that if $\mu$ has finite support, then the ancestral process, rescaled by the population size,
 converges to a time-changed 
Kingman-coalescent. Our first result shows that this remains true with the same classical scaling for $\mu$ with infinite support, as long as it has finite expectation.
We consider the ancestral process of a sample of $n\leq N$ individuals labelled $v_1,...,v_n$ sampled from generation $k=0.$ We define the equivalence relation $\sim_k$ on the set $\{1,...,n\}$ by
$$i\sim_k j\Leftrightarrow A(v_i)\cap A(v_j)\cap\big(\{-k,...,0\}\times\{1,...,N\}\big)\neq \emptyset,$$
that is $i\sim_k j$ if and only if $v_i$ and $v_j$ have a common ancestors at most $k$ generations back. Let $A_{N,n}(k)$ denote the set of equivalence classes with respect to $\sim_k,$ which is a stochastic process taking values in the partitions of $\{1,...,n\}.$ Let $E:=\{1,...,n\},$ and let $D_E[0,\infty)$ denote the space of c\`adl\`ag functions from $[0,\infty)$ to $E$ with the Skorohod topology.

\begin{theorem}
 \label{thm:convergence_to_kingman}
Assume \noemi{ $\E_\mu[\eta_1]<\infty.$ Let $\beta:=\frac{1}{\E_\mu[\eta_1]}.$} As $N\to\infty,$ the process 
$(A_{N,n}(\lfloor \frac{Nt}{\beta^2}\rfloor ))_{t\geq 0}$ converges weakly in $D_E[0,\infty)$ to Kingman's $n-$coalescent.
\end{theorem}

\noemi{Two individuals $v,w\in V_N$ have a common ancestor if and only if 
$A(v)\cap A(w)\neq \emptyset.$ If this is the case, and if
$v$ and $w$ belong to the same generation, we denote by $\tau$ the time to the most recent common ancestor,
$$\tau:=\inf\{n\geq 0: A(v)\cap A(w)\cap(\{-n\}\times\{1,...,N\})\neq \emptyset\}.$$
Clearly, the law of $\tau$ is the same for 
all $v,w$ with $i_v=i_w.$\\
Theorem \ref{thm:convergence_to_kingman} implies that if $\mu$ has finite expectation, two randomly sampled individuals have a common ancestor with probability 1, 
and the expected time to this ancestor is of order $N.$ If the expectation does not exist, this changes completely. Let us now assume that $\mu\in\Gamma_\alpha,$ which means that the tails of $\mu$ follow a power law. Our second result distinguishes three regimes:}

\begin{theorem}[Existence and expectation of the time to the most recent common ancestor]
\label{thm:tmrca}
Let $\mu\in\Gamma_\alpha$ and let $v,w\in V_N, v\neq w$
\begin{itemize}
 \item[(a)] If $\alpha\in(0,1/2),$ then $P_N^\mu(A(v)\cap A(w)\neq \emptyset)<1$ for all $N\in\N,$
\item[(b)] If $\alpha\in(1/2,1),$ then $P_N^\mu(A(v)\cap A(w)\neq \emptyset)=1$ and $E_N^\mu[\tau]=\infty$ for all $N\in\N.$
\item[(c)] If $\alpha>1,$ then $P_N^\mu(A(v)\cap A(w)\neq \emptyset)=1$ for all $N\in\N,$ and $\lim_{N\to\infty}\frac{E_N^\mu[\tau]}{N}=\frac{1}{\beta^2},$ with $\beta=\frac{1}{\E_\mu[\eta_1]}.$ 
\end{itemize}
\end{theorem}

In other words, for $\alpha>1/2$ two individuals almost surely share a common ancestor, but the expected time to the 
most recent common 
ancestor is finite for $\alpha>1$ and infinite if $\alpha\in(1/2,1).$ Hence in real-world populations observed over realistic time-scales,
for $\alpha\in(1/2,1)$ (or even for $\alpha\in(1,2)$ where the mean, but not the variance of $\mu$ exists), the assumption that a population is in equilibrium has to be treated with care. 

\begin{remark} In the boundary case $\alpha=1$, the choice of the slowly varying function $L$ becomes relevant. If we choose
$L=const.$, then it is easy to see from the proof that  $E_N^\mu[\tau]=\infty.$ The case $\alpha=1/2$ also depends on $L$ and
requires further investigation.
\end{remark}

\subsection{Forward in time process}\label{sect:forward}
Having obtained a good idea about the ancestral process, we would now like to study the forward picture. \noemi{For this it is useful to construct the whole bi-infinite genealogy of the whole population at once, which can be done as a spanning forest of a suitable vertex set.} We consider graphs -- in fact trees -- with vertex-set
$V_N=\Z^N$
and a set of bonds $E_N$ which will be a (random) subset of $B_N:=\{(v,w):v,w\in V_N\}$ where the edges are {\it directed}. 
For $v\in V_N$ we write as before $v=(i_v,k_v)$ with $i_v\in\Z,$ and $1\leq k_v\leq N.$ We consider the set of directed spanning 
forests of $V_N,$ which we can write down as follows: Let
$${\mathcal T}_N:=\{G=(V_N, E_N): E_N\subset B_N \mbox{ s.th. }\forall\,v\in V_N,\, \exists! \,
w\in V_N, i_w<i_v, \mbox{ with } e=(w,v)\in E_N\}.$$
This means, we consider trees where each vertex $v$ has exactly one outgoing (to the past) edge, which we denote by $e_v.$
This unique outgoing edge, or equivalently, the unique ancestor of $v$ is determined as follows. 
Let $\{\eta_v\}_{v\in V_N}$ be a countable family of independent $\mu-$ distributed random variables, and let $\{U_v\}_{v\in V_N}$ 
denote independent uniform random variables with values in $\{1,...,N\}$ independent of the $\eta_v.$ This infinite product measure induces 
a law on ${\mathcal T}_N$ if we define $$ e_v:=((i_v-\eta_v, U_v),v).$$ We denote this probability measure by $\hat{P}_N^\mu.$
In words, the ancestor of $v$ is found by sampling the generation according to $\mu,$ and then choosing the individual uniformly.
We see that 
\begin{equation}\label{eq:phat}\hat{P}_N^\mu(e_v=(w,v)\in E_N)=\frac{1}{N}\mu(i_v-i_w).\end{equation}
Comparing this to our previous construction of the ancestral process, we realise that  $P_N^\mu$ can be considered as being 
the restriction of $\hat{P}_N^\mu$
to situations regarding the ancestry of a sample, and hence, with slight abuse of notation, we will identify the two measures,
dropping the notation $\hat{P}_N^\mu.$ A tree $G\in{\mathcal T}_N$ is interpreted as the ancestral tree of the whole bi-infinite 
population. \noemi{
\begin{remark}
Note that for $\mu\in \Gamma_\alpha$ it follows from Theorem \ref{thm:tmrca} that $G\in \mathcal T_N$ has only one connected component almost surely if $\alpha>1/2,$ since two individuals belong to the same connected component if and only if their ancestral lines meet. If $\alpha<1/2,$ then $G$ has infinitely many connected components almost surely, since in that case any two individuals belong to two disjoint components with positive probability by Theorem \ref{thm:tmrca} (a). 
\end{remark}
Having obtained a construction of the genealogy of the population for all times, we can now for exampe introduce genetic types. We take the simplest situation of just two types. Let the individual $v\in V_N$ have type $X_v\in\{a,A\},$ and assume a neutral Wright-Fisher reproduction, that is, types are inherited from the parent. This means that in the above construction, individuals belonging to the same component of the tree have the same type. In particular, in the case $\alpha>1/2,$ everyone in the population has the same type. This is clear, since constructing the whole tree at once means that we are talking about a population in equilibrium, meaning that fixation of one of the two types has already occurred. However, in the case $\alpha<1/2$ the tree has infinitely many components almost surely, and therefore both types can persist for all times. We can assign to each component independently type $a$ with probability $p\in[0,1]$ and type $A$ otherwise. For each $p\in [0,1]$ this procedure defines a probability measure on $\{a,A\}^{V_N}.$
\begin{definition}
Let $\lambda_N^p$ denote the probability measure on $\{a,A\}^{V_N}$ which, given $G\in \mathcal T_N,$ assigns each connected component of $G$ independently type $a$ with probability $p,$ and type $A$ otherwise.
\end{definition} 
\begin{remark}
It can be shown, following \cite{HammondSheffield}, that in a certain sense the measures $\lambda_N^p$ are the only relevant probability measures on $\{a,A\}^{V_N}$ consistent with the dynamics of our population model. We make this precise in the appendix. For now, we just assume that the type distribution of our population is given by $\lambda_N^p.$
\end{remark}
We can now introduce the frequency process: Let $v:=(i,k)\in V_N,$ that is, $i$ denotes the genration of the individual, and $k$ its label among the $N$ individuals of generation $i.$ Let 
$$Y_N(i):=\frac{1}{N}\sum_{k=1}^N \mathbf{1}_{\{X_{i,k}=a\}}.$$
Our construction allows us to easily compute some correlations for the frequency process of the seed bank model.} Recall $q_n$ from the last section.

\begin{theorem}
\label{cor:correlations}Let $\lambda=\lambda_N^p.$ 
\begin{itemize}
 \item[(a)] $E_\lambda[Y_N(i)]=p \;\forall i\in\Z,$
\item[(b)] \noemi{If $\mu\in \Gamma_\alpha$ with} $\alpha>1/2,$ $\cov_\lambda(Y_N(0),Y_N(i))=p(1-p)\; \forall i\in\Z,\forall N\in\N,$
\item[(c)] If \noemi{$\mu\in \Gamma_\alpha$ with} $\alpha\in(0,1/2),$ we have $\lim_{N\to\infty}\cov_\lambda(Y_N(0),Y_N(i))=0,$
$$C(i):=\lim_{N\to\infty}{\mathrm{corr}}_{\lambda}(Y_N(0),Y_N(i))\in(0,1)\;\; \mbox{for all } i\in \Z,$$
and,
as $i\to\infty,$ for some constant $c$ and some slowly varying function $L,$ 
$$C(i)\sim \frac{ (1-\alpha)^2\cdot p(1-p)}{\Gamma(2-\alpha)^2\Gamma(2\alpha)\left(\sum_{n=0}^\infty q_n^2+1\right)}\cdot i^{2\alpha-1}L(i),$$
where $\sim$ means that the ratio of the two sides tends to 1, and the sum occurring in the denominator is finite.
\end{itemize}
\end{theorem}

\begin{remark}
 If $\alpha>1/2,$ we have that $\mathrm{corr}_\lambda(Y_N(0),Y_N(i))=1.$ \noemi{ This is clear, since in this case all individuals have the same type, and} $E_\lambda(Y_N(i))=p, \var_\lambda(Y_N(i))=p(1-p)$ and $\mathrm{corr}_\lambda((Y_N(0),Y_N(i))=1.$
\end{remark}

\section{Proofs}
\subsection{Proof of Theorem \ref{thm:convergence_to_kingman}}
\noemi{The proof of Theorem \ref{thm:convergence_to_kingman} follows ideas of \cite{KKL}, which we combine with a coupling argument relying on renewal theory. In certain steps we have to take particular care of the unboundedness of the support of the measure $\mu,$ these steps are carried out with particular care in Lemmas \ref{lem:stationary_measure} and \ref{lem:prob_coalescence}. Recall that for Theorem \ref{thm:convergence_to_kingman} we assumed that the expectation of the renewal process exists, i.e. $E_\mu[\eta_1]<\infty,$ which in the case $\mu\in \Gamma_\alpha$ holds for $\alpha>1.$} For the case $\alpha=1,$ finiteness of the expectation depends on the choice of the slowly varying function $L.$\\
We first introduce an `urn process' similar to the one introduced in \cite{KKL}, for measures $\mu$ with potentially unbounded support. The point is that
our ancestral process $A_N$ can then be realised as a simple function of this urn process.\\
Keep $N$ fixed. For $1 \le n \le N$ let 
$$
S_n:=\Big\{(x_1,x_2,...), x_i \in \mathbb{N}_0, \,\, \sum_{i=1}^\infty x_i=n\Big\}.
$$  
For $n\in\N$ we construct 
a discrete-time Markov chain $\{X^n(k)\}_{k\in\N_0}$ with values in $S_n$ that we will refer to as the $n-$sample process. Let 
$X^n(0)=(X_1^n(0), X_2^n(0),...)$ be such that $|X^n(0)|=n.$ We think of $X_i^n(0)\in \{0,..,n\}$ as the number of balls
 currently placed in urn number $i.$ Later, urns will correspond to generations, balls to individuals. The transition from time $k$ to
time $k+1$ is made by relocating the $X_1^n(k)$ balls in the first urn in a way that is consistent with the ancestral process of
our seed bank model, and shift the other urns including their contained balls one step to the left: Let $\sigma:\R^{\N}\to\R^{\N}:
(x_1,x_2,...)\mapsto (x_2,x_3,...)$ denote
 the one-step shift operator, and, for $l\in\N, $ let $R(l)$ be an $S_{l}-$valued
random variable which is multinomially distributed with infinitely many parameters:
$$R(l)\sim \mathrm{Mult}(l;\mu(1),\mu(2),...),$$
i.e. $R(i)$ is a random vector of infinite length, and $R_i(l)$ counts the number of outcomes that take value $i$ in $l$ independent
trials distributed according to $\mu.$ Define

\begin{eqnarray}
&& X^n(k+1) =\sigma \big(X^n(k)\big)+R(X^n_1(k)),\ \ \ k=0,1,\ldots
\end{eqnarray}

By definition, $X^n=\{X^n(k)\}_{k \in }$ is a Markov chain with (countably infinite) state space $S_n$ (see Figure 1).

\begin{figure}[h]
\begin{center}
 \includegraphics[scale=0.4]{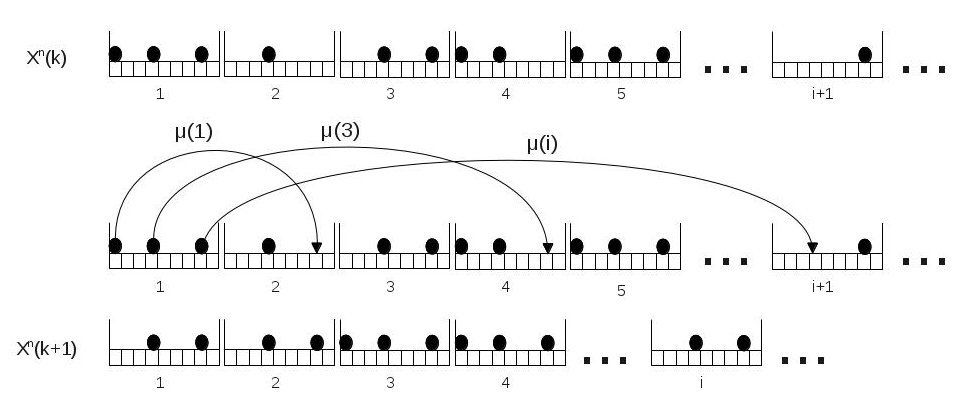}
\caption{\small Transition from $X^n(k)$ (top line) to $X^n(k+1)$ (bottom line): All the balls in urn number 1 are relocated independently according to $\mu.$}
\end{center}
\end{figure}

It provides a construction of $n$ independent
renewal processes with interarrival law $\mu$, if one keeps track of the balls. For our purpose, it suffices to note that $X_1^n(k)$ 
gives, for each $k,$ the number of renewal processes that have a renewal at after $k$ steps, which is equal in law  to the number of original individuals in our seed bank model that have 
an ancestor in generation $-k.$
Now recall our ancestral process $\{A_{N,n}(k)\}$ from Section 2,  which was constructed using coalescing renewal processes.
In terms of the $X^n-$process it can be described as follows: Think for the moment of each of the urns as being subdivided into $N$
sections. We start with $n$ balls and run the $X^n-$process. At each relocation step, each ball which is relocated to urn $i+1$
is put with equal probability into one of the $N$ sections in urn $i+1.$ All balls that end up in the same section within an urn 
are merged into a single ball (Figure 2).
 \begin{figure}[h]
\begin{center}
 \includegraphics[scale=0.4]{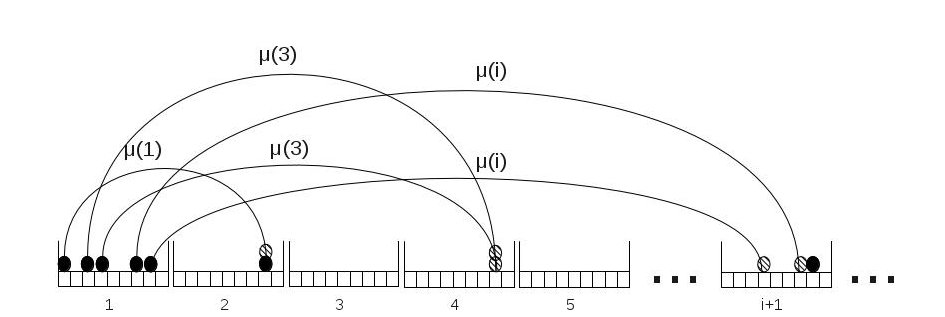}
\caption{\small The possible types of coalescence events in the $X^{N,n}-$process: A coalescent event in urn 2 induced by a ball
 landing in an occupied place, a coalescent event in urn 4 due to two balls landing in the same empty place,
 and no coalescence
in urn $i+1$ although it holds several balls.}
\end{center}
\end{figure}
Since this results in a decrease in the total number of balls, say from $n$ to $n' < n$, after a merger event, we continue 
to run accroding to a Markov process with law ${\cal{L}}(X^{n'})$ with $n'$ balls, and so on. Denote by $\{X^{N,n}(k)\}_{k\in\N}$ the well-defined process obtained by this procedure. 
The number of balls present at time $k$ in this
process is equal in law to the block-counting process of our ancestral process started with $n$ sampled individuals: 
$$|X^{N,n}(k)|\stackrel{d}{=}|A_{N,n}(k)|.$$

Unlike $A_N,$ the process 
$X^{N,n}=\{X^{N,n}(k)\}_{k\in\N}$ is a Markov chain in discrete time with countable state space $\cup_{i=1}^n S_i.$ Of course, it is also possible to define an 
exchangeable partition valued process as a function of $X^{N,n},$ where balls correspond to blocks (we refrain from a formal 
definition, in order to keep the notational effort reasonable).

\medskip

An important step is to observe that for each $n,$ the corresponding urn process $X^n$ has a unique invariant distribution. Indeed, let
$$
\beta_i:=\frac{\mu\{i,i+1,...\}}{E_\mu[\eta]}.
$$
This fraction is well-defined since we assumed $E_\mu[\eta]<\infty.$ Denote by $\nu^n:=\mathrm{Mult}(n,\beta_1,\beta_2,...)$ the multinomial
distribution with success probabilities $\beta_i.$ We claim that 
this is the stationary distribution for the $n-$sample process $X^n.$ From classical renewal theory, we know that $\nu^1$ is the 
stationary distribution in the case $n=1$ (see \cite{Lindvall_Book}). For $n$ independent renewal processes we have (cf. \cite{KKL}):

\begin{lemma} \label{lem:stationary_measure}
If $E_\mu[\eta]<\infty$, then $\nu^n$ is the stationary distribution for $X^n,$ and $X^n$ is positive
recurrent for all $n\in\N.$
\end{lemma}

\begin{proof}

We reduce the proof to the finite case discussed in \cite{KKL}. For each $j\in\N$ we define  
$$\mu_j(\{i\}):=\frac{1}{\sum_{l=1}^j\mu(\{l\})}1_{\{i\leq j\}}\mu (\{i\}),\quad i\in\N.$$
This defines a probability measure $\mu_j$ with support $\{1,...,j\}.$
Clearly, $\lim_{j\to\infty}\mu_j(i)=\mu(i)$ for all $i,$ and $\lim_{j\to\infty}\E_{\mu_j}[\eta]=E_\mu[\eta]$ by monotone convergence.

Let $Y^{n,j}=(Y^{n,j}(k))_{k\in\N_0}$ be the Markov chain constructed in the same way as $X^n,$ but with relocation
 measure $\mu_j$ instead of $\mu,$ that is,
$Y^{n,j}(k+1)=\sigma( Y^{n,j}(k))+R^j(Y^{n,j}_1(k)),$
where $R^j(l)\sim  \mathrm{Mult}(l;\mu_j(1),...,\mu_j(j)),$ and with $Y^{n,j}(0)=X^n(0).$ 
Define now
$$\beta^j_i:=\frac{\mu_j\{i,i+1,...\}}{E_{\mu_j}[\eta]}.$$ 
Clearly, $\lim_{j\to\infty}\beta_i^j=\beta_i\ \ \forall i\in\N.$
Let $\nu^n_j:=\mathrm{Mult}(n;\beta^j_1,\beta^j_2,...)$ the multinomial
 distributions
on $S_n$ with success probabilities $\beta_i^j.$ 
By Lemma 1 of \cite{KKL} we know that $\nu_j^n$ is the stationary distribution for $Y^{n,j}$. Fix $x,y\in S_n.$ By construction,
\begin{equation}\begin{split}P(X^n(1)=y\,|\, X^n(0)=x)=&P(R(x_1)=y-\sigma(x))=\lim_{j\to\infty}P(R^j(x_1)=y-\sigma(x))\\      
=&\lim_{j\to\infty}P(Y^{n,j}(1)=y\,|\, Y^{n,j}(0)=x).\end{split}
\end{equation}
For $x\in S_n,$ let $j_x:=\max\{j: x_j\neq 0\}.$ Note $P(X^n(1)=y\,|\,X^n(0)=x)=0$ for all $x$ such that $j_x>j_y+1.$ We write 
$P_{\nu^n}$ for the distribution of $(X^n(k))_{k\in\N}$ with initial distribution $\nu^n.$ Then, for every $y\in S_n,$
\begin{align}
P_{\nu^n}(X^n(1)= y) & = \sum_{x\in S_n, j_x\leq j_y+1}\nu^n(x)P(X^n(1)=y\,|\, X^n(0)=x)\\
&=\lim_{j\to\infty}
\sum_{x\in S_n, j_x\leq j_y+1}\nu^n_j(x)P(Y^n(1)=y\,|\, Y^{n,j}(0)=x)\\&= \lim_{j\to\infty}\nu^n_j(y)=\nu^n(y).
\end{align}
So $\mathrm{Mult}(n;\beta_1,\beta_2,...)$ is a stationary distribution for $X^n$. By irreducibility it is unique, and
 $X^n$ is positive recurrent.
 \end{proof}

Recall the dynamics of the process $X^{N,n}=(X^{N,n}(k))_{k\in\N_0}$ from above. 
We first compute the probability of a coalescence given that we are in a fixed configuration. Define the events
$$B_{l,k}:=\{\mbox{exactly }l \mbox{ mergers at time }k\mbox{ in }X^{N,n}\}$$
and $$B_{\geq l,k}:=\{\mbox{at least }l \mbox{ mergers at time }k\mbox{ in }X^{N,n}\},$$
for $1 \le l \le n$ and $k\in\N.$
\begin{lemma}
 \label{lem:prob_coalescence}Fix $N\in\N,$ $n<N,$ and $\mu$ such that $\E_\mu[\eta]<\infty.$ With the notation of the last section,
\begin{equation}\begin{split}
P\big(B_{1,k+1}\,\big|\,X^{N,n}(k)=(x_1,x_2,...)\big)=\frac{1}{N}\sum_{i=1}^\infty \left(x_1x_{i+1}\mu(i)+
{{x_{1}}\choose{2}}\mu(i)^2\right)+O(N^{-2})\end{split}\end{equation}
and there exists $0< c(n)<\infty,$ depending on $X^{N,n}$ only via $n,$ such that
$$P\big(B_{\geq 2,k+1}\,\big|\,X^{N,n}(k)=(x_1,x_2,...)\big)\leq \frac{c(n)}{N^{2}}.$$
\end{lemma}

\begin{proof}
We start with computing
the probability of a coalescence in a fixed urn $i\in\N$ given $X^{N,n}(k)=(X^{N,n}_{1}(k), X^{N,n}_{2}(k),...)$ and $R(X^{N,n}_1(k))=(R_1(X^{N,n}_1(k)), R_2(X^{N,n}_1(k)),...)$.
 The probability for having {\em exactly} one coalescence occurring in urn $i$ 
(note that from $k$ to $k+1$ we shift all urns by 1) is
$$\frac{1}{N}X^{N,n}_{i+1}(k)R_i(X^{N,n}_1(k))+\frac{1}{N}{{R_i(X^{N,n}_1(k))}\choose{2}}-p(i),$$
where $p(i)=p(i,X^{N,n}(k),R(X_1^{n,N}(k)))$ is the probability that more than one coalescence happens in urn $i.$ Here, the first term is the probability that we see at least one coalescence due to one of the
 relocated balls falling into an already occupied section of urn $i$, and the second term
 is the probability of seeing at least one coalescence due to two relocated balls falling into the same section of  urn $i$.
Observe that
$p(i)$ is $O(N^{-2})$. More precisely,
writing 
$$M_i:=X^{N,n}_{i+1}(k)R_i(X^{N,n}_1(k))+{{R_i(X^{N,n}_1(k))}\choose{2}},$$
it is easy to see that, because each ball being moved to urn $i$ has a probability of at most $\frac{n}{N}$ to merge at all,
$$p(i)\leq \frac{n^4}{N^2},$$
and therefore, since \emph{given} $X^{N,n}(k)$ and $R(X_1^{N,n}(k))$ there are at most $n$ occupied urns,
$$\sum_{i=1}^\infty p(i)\leq \frac{n^5}{N^2}.$$
Further, given $X^{N,n}(k)$ and $R(X_1^{N,n}(k)),$ the probability of having at least two mergers at step $k+1,$ which occur in two different
 urns $i$ and $j,$ is 
\begin{equation*}\begin{split}
\frac{1}{N^2} M_i\cdot M_j.\end{split}\end{equation*}
Moreover, for fixed $X^{N,n}(k)$ and $R(X_1^{N,n}(k)),$ we have the trivial bound $\sum_{j=1}^\infty M_j\leq 2n^3.$
This implies 
$$\frac{1}{N^2} \sum_{i=1}^\infty \sum_{j:j\neq i}M_i\cdot M_j\leq \frac{4n^6}{N^2}.$$
Thus the probability of seeing exactly one coalescence in step ${k+1},$ \emph{given} $X^{N,n}(k)$ and $R(X^{N,n}_1(k)),$ is
$$\sum_{i=1}^\infty \left(\frac{1}{N}M_i-p(i)\right)-\frac{1}{N^2}\sum_{{i,j=1}\atop{j\neq i}}^\infty M_iM_j=
 \frac{1}{N}\sum_{i=1}^\infty M_i+O(N^{-2}).$$
Computing $R(X^{N,n}_1(k))$ given $X^{N,n}(k)$ using the multinomial distribution, we obtain
\begin{equation}\begin{split}
P\big(B_{1,k+1}\,\big|\,X^{N,n}(k)=x\big)=&\sum_{r\in S_n}P(B_{1,k+1}\,|\,X^{N,n}(k)=x, R(x)=r)P(R(x)=r\,|\,X^{N,n}(k)=x)
\\
=&\frac{1}{N}\sum_{r\in S_n}\left[\sum_{i=1}^\infty \left(x_{i+1}r_i+
 {{r_i}\choose{2}}\right)+O(N^{-2})\right]P(R(x)=r\,|\,X^{N,n}(k)=x)\\
= &\frac{1}{N}\sum_{i=1}^\infty \left( x_{i+1}x_{1} \mu(i)+
{{x_{1}}\choose{2}}\mu(i)^2\right)+O(N^{-2}),\end{split}\end{equation}
where we have used that 
$$\sum_{r\in S_n}O(N^{-2})P(R(X_1^{N,n})=r\,|\,X^{N,n}(k)=x)=O(N^{-2})$$ since the $O(N^{-2})$ term 
is bounded uniformly in $r\in S_n$ by some $\frac{c(n)}{N^2},$ and we average with respect to a probability measure.
This proves the first claim. We have seen that 
\begin{equation*}P\big(B_{\geq 2,k+1}\,\big|\,X^{N,n}(k), R(X^{N,n}_1(k)\big)=\sum_{i=1}^\infty p(i)+  
 \frac{1}{N^2}\sum_{{i,j=1}\atop{j\neq i}}^\infty M_i\cdot M_j\leq \frac{c(n)}{N^2}.\end{equation*}
This proves the second part.

\end{proof}

%

We now have the ingredients to prove convergence to Kingman's coalescent.\\
%

{\bf Proof of Theorem \ref{thm:convergence_to_kingman}.} Fix $n\in \N.$ We will first study the process started in
 the stationary distribution $\nu^.$ Then we will extend the 
result to arbitrary initial distributions using an adaptation of Doeblin's coupling method.
To prove convergence in the stationary case, we just need to prove that the inter-coalescence times for binary mergers are distributed 
asymptotically exponential with rate $\beta_1^2{{n}\choose{2}},$ and that multiple coalescences are negligible. Starting from the stationary distribution, the probability of seeing a coalescence in the 
next step given that we have currently $n$ balls is \noemi{obtained as in \cite{KKL}, using Lemma \ref{lem:prob_coalescence}:}

\begin{equation}\begin{split}
P\big(B_{1,k+1}\,\big|\, X^{N,n}(k)\sim \nu^n\big) 
&=\E_{\nu^n}\big[P\big(B_{1,k+1}\,\big|\, X^{N,n}(k)\big)\big]\\
&=\frac{\beta^2_1}{N}{{n}\choose{2}}\left(2\sum_{i=1}^{\infty}
\frac{\beta_{i+1}}{\beta_1}\mu(i)+\sum_{i=1}^\infty \mu(i)^2\right)+O(N^{-2})\\
 &=\frac{\beta^2_1}{N}{{n}\choose{2}}+O(N^{-2}),
\end{split}\end{equation}

where we have computed the expectations with respect to the multinomial distribution $\nu^n$ and used  $2\sum_{i=1}^{\infty}
\frac{\beta_{i+1}}{\beta_1}\mu(i)+\sum_{i=1}^\infty \mu(i)^2=1.$

We have seen before that multiple coalescences happen with
negligible probability.
 Hence if we speed up time by a factor $N,$ we obtain for the inter-coalescence times
\begin{equation}
\label{convergence_times}
 \begin{split}
  \lim_{N\to\infty}P(\mbox{no coalescence in } X^{N,n} \mbox{ before time }Nt)=\lim_{N\to\infty}\left(1-\frac{\beta^2_1}{N}{{n}\choose{2}}+O(N^{-2})\right)^{Nt}
=e^{-\beta^2_1{{n}\choose{2}}t}.
 \end{split}
\end{equation}
For the coupling argument, we consider now a process $\tilde{X}^{N,n}$ which runs as follows: Start with $n$ balls in the stationary
 distribution $\nu^n,$ and let it evolve
according to the $n-$sample dynamics. After each coalescence event, sample a new starting configuration according to $\nu^{n'},$ 
where $n'$ is the number of balls present after the coalescence, and run the process according to the $n'-$sample dynamics.
Assume now that $X^{N,n}$ starts in a given initial distribution.
Define
$$T^{(N)}:=\inf\{t>0: X^{N,n}(t)=\tilde{X}^{N,n}(t)\}.$$
We couple $X^{N,n}$ and $\tilde{X}^{N,n}$ as follows. Colour the balls of $X^{N,n}$ red and the balls of $\tilde{X}^{N,n}$ blue.
Label both the red and the blue balls $1,...,n.$ Recall that the dynamics of our urn process just consists in moving balls from 
urn one independently from each other to a new urn according to $\mu,$ and merging balls in the same urn with probability $\frac{1}{N}$ per pair.
Run the red and the blue process independently.
Let us first assume that no coalescences occur in either of the processes. \\
Now if at some time $k,$ the red ball number $i$ and the blue ball number $i$ happen to be 
in the same urn (but not necessarily in the same section), we couple them and let them move together from this time onwards. 
Denote by $\sigma_i$ the time of this
coupling. Note that $\sigma_i$ is finite almost surely, since it is the coupling time of two renewal processes .
Then we continue running our processes until all the balls have coupled. Let $T_{coup}:=\max\{\sigma_i, 1\leq i\leq n\}.$ Note 
that this time is independent of $N.$
Since $n$ is fixed, and the different balls move independently, we have $P(T_{coup}<\infty)=1$ no matter which initial distributions
we choose (see
\cite{Lindvall_Book}, chapter II), and hence
$$\lim_{t\to\infty}P(T_{coup}\geq t)=0.$$
Speeding up time by $N,$ the coupling happens much faster than the coalescence: Let $T^{(N)}_{coal}$ be the time of the 
first coalescence in either the red or the blue process. At each time step, the probability of having a coalescence in the next step
is bounded from above by the crude uniform estimate $n^2/N.$ Hence 
$$\lim_{N\to\infty}P\big(T^{(N)}_{coal}\geq  \sqrt{N}\big)\geq \lim_{N\to\infty}\left(1-\frac{n^2}{N}\right)^{\sqrt{N}}=1.$$
Since
$$\lim_{N\to\infty}P\big(T_{coup}\leq \sqrt{N}\big)=1,$$
we get
$$\lim_{N\to\infty}P\big(T^{(N)}_{coal}\geq T_{coup}\big)\geq \lim_{N\to\infty}P\big(T^{(N)}_{coal}\geq \sqrt{N},T_{coup}\leq \sqrt{N}\big)=1.$$
This implies
$$\lim_{N\to\infty}P\big(T^{(N)}\neq T_{coup}\big)=\lim_{N\to\infty}P\big(T^{(N)}_{coal}<T_{coup}\big)=0,$$
from which we see 
$$ \lim_{N\to\infty}P\big(T^{(N)}\geq Nt\big)=\lim_{N\to\infty} P\big(T_{coup}\geq Nt\big)=0.$$
Hence we can restart our process $\tilde{X}^{N,n}$ after each coalescence event, and the two processes will couple with probability 1 before the next coalescence 
takes place, and indeed on the coalescent time scale (time sped up by $N$) the coupling happens instantaneously.
Using (\ref{convergence_times}) we thus obtain for the inter-coalescence times of the process started in an arbitrary but fixed initial configuration
\begin{equation}
 \begin{split}
\lim_{N\to\infty}P(\mbox{no coalescence in } X^{N,n} \mbox{ before }Nt)
=&e^{-\beta^2_1{{n}\choose{2}}t}.
\end{split}
\end{equation}
This implies as before by standard arguments that $|X^{N,n}(Nt)|$ converges weakly as $N\to\infty$ to the block-counting process of 
Kingman's coalescent. Since $|X^{N,n}(Nt)|\stackrel{d}{=}|A_{N,n}(Nt)|,$ and the fact we obviously have exchangeability of the 
ball configurations,
we even obtain the convergence to Kingman's $n-$coalescent in the obvious sense.\hfill $\Box$

\begin{remark}
It appears remarkable that $\E_{\mu}[\eta]<\infty$ is sufficient for this result.
 If $\E_{\mu}[\eta^2]=\infty,$ and $Y$ denotes the label of the urn that a ball is placed in, then $\E_{\nu^n}[Y]=\infty$
 and by \cite{Lindvall}, $\E[S]=\infty.$ However, due to the
time rescaling, the fact that $P(S<\infty)=1$ is enough for our purpose.
\end{remark}

\begin{remark}\label{proof_moehle}
 In order to show convergence to Kingman's coalescent, we could also
 follow the approach of \cite{KKL}, which uses M\"ohle's Lemma \cite{Moehle} to show convergence of finite dimensional distributions. 
Note however that in our case the state space of the Markov chain is infinite, hence the transition matrices are infinite.
Indeed, denoting the transition matrix of $X^{N,n}$ by 
$\Pi_N=\{\Pi_N(x,y)\}_{x,y\in\cup_{j=1}^\infty S_j},$ we can decompose $\Pi_N$ as 
$\Pi_N=A+\frac{1}{N}B+O(N^{-2}),$
where $A$ is given by the transitions of the $X^n-$processes without coalescence, and $B$ contains adjustments that need to 
be made 
to the $X^n-$process in case of a single coalescence event (compare \cite{KKL}). The higer order coalescences are $O(N^{-2})$ by
Lemma \ref{lem:prob_coalescence}. To apply M\"ohle's Lemma 
it is sufficient to show that $P:=\lim_{m\to\infty}A^m$ 
and $G:=PBP$ exist. 
We first take care of the part without coalescence. Let $A$ be defined by
$A(x,y):=\sum_{j=1}^n1_{\{x,y\in S_n\}}A_n(x,y),$
where $(A_n(x,y))_{(x,y)\in S^n}$ denotes the transition matrix of $X^n.$ Then Lemma \ref{lem:stationary_measure} yields
$\lim_{k\to\infty}A_n^k(x,y)=\nu^n(y)$ for all $x,y\in S_n.$
Therefore we obtain 
$\lim_{m\to\infty}A^m=P,$
where $P=(P(x,y))_{x,y\in S}$ with $P(x,y)=\sum_{j=1}^n1_{\{x,y\in S_j\}}\nu^j(y).$
We can now define $B$ as the matrix of the single coalescence events as in \cite{KKL}. That is, if
 $x\in S_i,y\in S_{i-1},$ then
$B(x,y)$ is the probability that the balls from configuration $x$ are relocated according to the matrix $A_i,$ and that exactly one pair of them
 coalesces, so that we end up with configuration $y.$ If $x\in S_i,$ then $B(x,y)=0$ if $y\notin S_i\cup S_{i-1}.$ If $x$ and 
$y$ are in $S_i,$ then $B(x,y)$ gives the correction for the $X^n-$process in case of a coalescence, therefore
 $B(x,y)\geq -A(x,y)$ in this case. Hence $B$ has the same block form as in \cite{KKL}, however, the single blocks are 
of infinite size. Furthermore,
$\|B\|=\max_{x\in \cup_{i=1}^nS_i}\sum_{y}|B(x,y)|\leq 2.$ Since $P$ is a projection, $G=PBP$ is a bounded operator, 
and therefore $e^{tG}, t\in\R,$ exists as a convergent series. Now the computations work in exactly as in the case of 
bounded support, hence we obtain the convergence to Kingman's coalescent following the proof of \cite{KKL}.\hfill 
$\Box$
\end{remark}

\begin{remark} Note that M\"ohle's result allows the following heuristic interpretation of our limiting process $X^{N, n}$ as 
$N\to\infty.$
 First, the process, for each number of `active' balls $n' \le n$, mixes rapidly and essentially instantaneously enters its stationary distribution
on the configuration with $n'$ balls. Note that as long as there is no coalescence event, any future evolution does not 
affect the block counting process $A_{N,n}$, and also not the corresponding partition-valued process, where each `active' ball denotes a block
in a partition of $\{1, \dots n\}$ consisting of all labels of balls that have merged into this active ball.
Now, in each `infinitesimal time step', our limiting process picks an entirely new state from its stationary distribution, independent of its `previous' state (this is the effect of the projection operator $P$).
In a way it can be regarded as a `white noise' process on the space of stationary samples.
While this process obviously has no c\`adl\`ag modification, both the block counting process, and the partition valued process, remain constant until there is a new merger, and are thus well-defined (recalling that such
mergers, that is,  transitions from $n'$ active balls to $n'-1$ active balls, happen at finite positive rate in the limit).
\end{remark}

\subsection{Proof of Theorem \ref{thm:tmrca}}
Recall from section 2
 that the time to the most recent common ancestor is related to the
coupling time of two versions of the renewal process. Recall 
$$q_n=P_N^\mu\Big(A(v)\cap\big(\{-n\}\times\{1,...,N\}\big)\neq \emptyset\Big).$$
We will need some bounds on the $q_n$ that can be obtained via Tauberian
theorems. 

\begin{lemma}
 \label{lem:bounds}
Let $\mu\in\Gamma_\alpha.$ 
\begin{itemize}
 \item[(a)] Let $\alpha\in(0,1).$ Then
$$\sum_{n=0}^{i} q_n\sim\frac{1-\alpha}{\Gamma(2-\alpha)\Gamma(1+\alpha)}\cdot i^\alpha L(i)^{-1}\mbox{ as }i\to\infty,$$
\item[(b)] The sum
$$\sum_{n=0}^\infty q_n^2$$
is finite if $\alpha\in(0,1/2)$ and infinite if $\alpha>1/2.$
\item[(c)] Let $\alpha\in(0,1/2).$ Then
$$\sum_{n=0}^\infty q_nq_{n-i}\sim\frac{(1-\alpha)^2}{\Gamma(2-\alpha)^2\Gamma(2\alpha)}\cdot i^{2\alpha-1}L(i)\mbox{ as }i\to\infty.$$
\end{itemize}
\end{lemma}
{\bf Proof.} The proof of this lemma can be found in \cite{HammondSheffield}, Lemma 5.1.\hfill $\Box$\\

{\bf Proof of Theorem \ref{thm:tmrca}.} We first prove $(c),$ which corresponds to the case where we have convergence 
to Kingman's coalescent. Without loss of generality, assume $i_v=i_w=0.$ Denote by $(R_n)$ and
$(R'_n)$ the sequences of renewal times of the renewal processes corresponding to $v$ and $w$ respectively,\noemi{ that is, $R_n=\mathbf{1}_{\{n\in \{S_0,S_1,...\}\}}.$} In other words,
$R_n=1$ if and only if $v$ has an ancestor in generation $-n,$ \noemi{and $q_n=P(R_n=1).$} Let
$$T:=\inf\{n: R_n=R_n'=1\}$$
denote the coupling time of the two renewal processes. Since each time $v$ and $w$ have an ancestor in the same generation, these
ancestors are the same with probability $N,$ we get
$$E[\tau]=NE[T].$$
But if $\alpha>1,$ we have that $E_\mu[\eta_1]<\infty,$ and therefore by Proposition 2 of \cite{Lindvall}, $E[T]<\infty.$ 
The result now follows from Theorem \ref{thm:convergence_to_kingman} and the fact that the
expecte time to
the most recent common ancestor of $n$ individuals in Kingman's coalescent \noemi{with time change $\beta^2$} is given by
$$E[T_{MRCA}]=\frac{1}{\beta^2}\sum_{k=2}^n\frac{1}{{k \choose 2}}=\frac{2}{\beta^2}\left(1-\frac{1}{n}\right),$$
\noemi{hence for $n=2$ we get $\frac{1}{\beta^2}.$}

$(b)$ For independent 
samples $R$ and $R',$ 
the expected number of generations where both individuals have an ancestor, is given by
$$E\big[\sum_{n=0}^\infty R_nR'_n\big]=\sum_{n=0}^\infty E[R_n]E[R'_n]=\sum_{n=0}^\infty q_n^2,$$
which is infinite if $\alpha>1/2$ due to Lemma \ref{lem:bounds} $(b)$. Each of these times, the ancestors are the same with
probability $1/N,$ therefore with probability one $A(v)$ and $A(w)$ eventually meet. However, the expected time until this event
is bounded from below by the expectation of the step size,
$$E_\mu^N[\tau]\geq E[\eta]=\infty$$
if $\alpha<1.$\\
$(a)$ In this case, $E\big[\sum_{n=0}^\infty R_nR'_n\big]=\sum_{n=0}^\infty q_n^2<\infty,$ and therefore 
$$P\big(\sum_{n=0}^\infty R_nR'_n=\infty\big)=0,$$ which implies that the probability that $A(v)$ and $A(w)$ never meet is positive.

\subsection{Proof of Theorem \ref{cor:correlations}}
We prove  now Theorem \ref{cor:correlations}. We define $Y_v:=1_{\{X_v=a\}}. $
\begin{lemma}
\label{lem:cov} Let $\lambda=\lambda_N^p,$ \noemi{and assume $\mu\in \Gamma_\alpha.$}
\begin{itemize}
 \item[(a)] If $\alpha>1/2,$
$$\cov_\lambda(Y_v,Y_w)=p(1-p),$$
\item[(b)]If $\alpha\in (0,1/2), v\neq w,$
$$\cov_\lambda(Y_v,Y_w)=p(1-p)\frac{\sum_{n=0}^\infty q_nq_{n+i_v-i_w}}{N+\sum_{n=1}^\infty q_n^2}.$$
\end{itemize}
\end{lemma}
{\bf Proof.} We have
$$E_\lambda^N(Y_vY_w)=\lambda(X_v=X_w=a)=pP_N^\mu\big(A(v)\cap A(w))\neq \emptyset\big)+p^2\big(1-P_N^\mu\big(A(v)\cap A(w)\neq\emptyset\big)\big)$$
and $E_\lambda^N(Y_v)E_\lambda^N(Y_w)=p^2.$ This implies

$$\cov_\lambda(Y_v,Y_w)=p(1-p)P^N_\mu\big(A(v)\cap A(w)\neq \emptyset\big).$$
If $\alpha>1/2,$ then $P_N^\mu(A(v) \cap A(w)\neq \emptyset)=1$ which proves $(a).$ 
Hence we need to compute $P_N^\mu(A(v) \cap A(w)\neq \emptyset)$ for $\alpha<1/2.$ 
To do this, let $S_n,S'_n$ denote two independent samples of the renewal process, with $S_0=i_v, S'_0=i_w.$ Note that this 
implies for the times of the renewals that 
$$P(R_n=1)=q_{n+i_v}.$$
Recall that the renewal process is running forward in time, whence the ancestral lines are traced backwards.
Let $A_v$ and $A_w$ denote two independent samples of the ancestral lines of $v$ and $w,$ using the processes $S$ and $S'$ 
respectively, without coupling the processes. Then the expected number of intersections of $A_v$ and $A_w$ is given by
\begin{equation}
 \begin{split}
 E[|A_v\cap A_w|]=\frac{1}{N} E\big[\sum_{n=-i_w}^\infty R_nR'_n\big]=&\frac{1}{N}\sum_{n=-i_w}^\infty q_{n+i_v}q_{n+i_w}\\
=&\frac{1}{N}\sum_{n=0}^\infty q_{n}q_{n+i_v-i_w},
 \end{split}
\end{equation}
On the other hand, conditioning on the event that the ancestral lines meet (which clearly has positive probability), 
and then restart the renewal processes in the 
generation of the first common ancestor, which is the same as sampling two ancestral lines starting at $(0,0),$
\begin{equation*}\begin{split}
  E[|A_v\cap A_w|]= &E\big[|A_v\cap A_w|\,\big| \,A_V\cap A_w\neq \emptyset\big]
P\big(A_v\cap A_w\neq \emptyset\big)\\
=&P(A(v)\cap A(w)\neq \emptyset)E\big[|A_{(0,0)}\cap A_{(0,0)}|]\\
=&P(A(v)\cap A(w)\neq \emptyset)\left(q_0+\frac{1}{N}\sum_{n=1}^\infty q_n^2\right).                  
         \end{split}
\end{equation*}
Recalling $q_0=1$ this implies
$$P_N^\mu(A(v) \cap A(w)\neq \emptyset)=\frac{\sum_{n=0}^\infty q_nq_{n+i_v-i_w}}{N+\sum_{n=1}^\infty q_n^2},$$
which proves the Lemma. 
\hfill $\Box$\\

\noindent
{\bf Proof of Theorem \ref{cor:correlations}.} $(a)$ is obvious and $(b)$ follows from Lemma \ref{lem:cov}.
For $(c),$ let $\alpha\in (0,1/2).$ Lemma \ref{lem:bounds} tells us that
 $\sum_{n=0}^\infty q_n^2<\infty.$ From Lemma \ref{lem:cov} it follows that for $i\neq 0,$
$$\cov_\lambda(Y_N(0),Y_N(i))=p(1-p)\frac{\sum_{n=0}^\infty q_nq_{n-i}}{N+\sum_{n=1}^\infty q_n^2}.$$
For the variance we obtain
\begin{equation*}\begin{split}
\var_\lambda(Y_N(i))=&\frac{1}{N^2}\sum_{k,j=1}^N\cov_\lambda(Y_{(i,k)},Y_{(i,j)})\\
=& \frac{1}{N^2}\left(N p(1-p)+N(N-1)p(1-p)\frac{\sum_{n=0}^\infty q_n^2}{N+\sum_{n=1}^\infty q_n^2}\right)\\
=&p(1-p)\frac{\sum_{n=0}^\infty q_n^2+1-1/N}{N+\sum_{n=1}^\infty q_n^2}.
\end{split}\end{equation*}
Hence
$$\mathrm{corr}_\lambda(Y_N(0),Y_N(i))=\frac{\sum_{n=0}^\infty q_nq_{n-i}}{\sum_{n=0}^\infty q_n^2+1-1/N}$$
which converges as $N\to\infty.$
The result now follows from Lemma \ref{lem:bounds} $(c).$
 \hfill $\Box$\\

\begin{appendix}

\section{Appendix: Gibbs measure characterization of the forward process}
\noemi{In section \ref{sect:forward} we claimed that the measures $\lambda_N^p$ are in a certain sense the only measures describing the type distribution which are consistent with the dynamics of our process. In order to make this rigorous, we use a Gibbs measure characterization, which relies on the approach of \cite{HammondSheffield}. }
In order to construct the Gibbs measure, we start with prescribing the distribution of types  conditional on the 
(infinite)
past. Let $S_N:=\{a,A\}^N$ denote the finite dimensional state space. Let $X_v=X_{(i_v,k_v)}\in \{a,A\}$ denote the type of 
individual $v$ that is the $k$th individual of generation $i.$ We denote by ${\mathcal C}$ the sigma-algebra of cylinder events,
and write $\sigma_{n}$ for the $\sigma-$algebra generated by cylinder sets contained in $\{....,n\}.$
For $i\in\Z,$ we define the probability kernel $\lambda_{N,i}(\cdot|\cdot)$ from $(S_N^\Z,\sigma_{i})$ to $(S_N^\Z, {\mathcal C})$
 by saying that for any finite set $B\subset \{i+1,...\}^N,$ and 
$x_B\in\{a,A\}^B,$ and for $\xi\in S_N^{\{...,i-1,i\}}$ the conditional probability
$$\lambda_{N,i}^\xi(X\arrowvert_B=x_B):=\lambda_{N,i}(\{X\arrowvert_{B}=x_B\}\;|\;\xi)$$
is obtained by first sampling $G\in {\mathcal T}_N,$ 
tracing back the ancestral line of every $v\in B$ until
it first hits $\{...,i\},$ and then assigning the type $\xi_{\cdot}$ of this ancestor to $v.$ This is well defined because 
under $P_N^\mu$ the tree until it first hits $\{...,i\}$ is independent of $\sigma_{i}.$
These kernels $\lambda_{N,i}^\xi, i\in\Z$ are now used to construct the
 Gibbs measures. Due to the construction via product measures it is clear that they are consistent: If $i<j,$ then for
 $B\subset\{j+1,...\}\times \{1,...,N\},$
$$\lambda_{N,i}^{\xi^1}(X_{v}=x_v, v\in B\,\mid\, X_{w}=\xi_w^2, i+1\leq i_w\leq j)=\lambda_{N,j}^{\xi^1\vee\xi^2}
(X_v=x_v, v\in B).$$
Here, $\xi^1\vee\xi^2$ denotes the configuration which is equal to $\xi^1$ on $\{...,i\}$ and equal to $\xi^2$ on $\{i+1,...,j\}.$
So we can now define the Gibbs measures for our model:

\begin{definition}\label{def:mu-gibbs}
A probability measure $\lambda_N$ on $S_N^\Z$ is called a {\bf $\mu-$Gibbs measure} if for all 
 $i\in\Z,$ for all finite subsets $B\subset \{i+1,...,\}\times\{1,...,N\},$ and for all $x_B\in \{a,A\}^B$ the 
mapping $\xi\mapsto \lambda_{N,i}^\xi(x_B)$ is a version of the conditional probability 
$$\lambda_N(X\arrowvert_{B}=x_B\;\mid\;\sigma_{i}).$$
\end{definition}

In other words, to sample from the Gibbs measure {\it conditional on the past} up to generation $i,$ we first sample a $G\in{\mathcal T}_N$
according to $P_N^\mu,$ and assigning each $X_v, i_v\geq i+1, 1\leq k_v\leq N,$ its type according to the ancestors.\noemi{ It is clear
that such measures exist, in fact, $\lambda_N^p$ defined in section \ref{sect:forward} clearly is a $\mu-$Gibbs measure for $p\in\{0,1\},$ and if
$G\in{\mathcal T}_N$ has infinitely many components almost surely, then for all $p\in[0,1],$ the measures $\lambda_N^p$ are 
$\mu-$Gibbs measures. Recall that this is the case if $\mu\in \Gamma_\alpha$ with $\alpha<1/2.$ This is the situation where the Gibbs measure characterization is interesting.}\\
A particularly useful feature of our model is that the only relevant Gibbs measures are of the form $\lambda_N^p.$ Note that the 
$\mu-$Gibbs measures form a convex set, as can be seen easily, and we can characterise the extremal points of this set generalizing
Proposition 1 of \cite{HammondSheffield}.

\begin{proposition}
 \label{thm:char_gibbs_meas} \noemi{Assume $\mu\in \Gamma_\alpha.$}
\begin{itemize}
 \item [(a)]Let $\alpha\in(0,1/2).$ For each fixed $N,$ for each $p\in[0,1],$ there is precisely one extremal $\mu-$Gibbs measure $\lambda_N$ on
 $S_N^\Z$ such that
$\lambda_N(X_{i,k}=a)=p$ for all $i\in\Z, 1\leq k\leq N.$
\item[(b)] Let $\alpha\in(1/2,\infty].$ The only extremal Gibbs measures are $\lambda_N^0$ and $\lambda_N^1.$ For $p\in(0,1),$ the
measures $\lambda_N^p$ are given by $\lambda_N^p=p\lambda_N^0+(1-p)\lambda_N^1.$
\end{itemize}
\end{proposition}

The proof of Proposition \ref{thm:char_gibbs_meas} follows closely the Proposition 1 of \cite{HammondSheffield}, and
we refer the reader to this work for details. Note that part $(b)$ follows immediately from Theorem \ref{thm:tmrca}, as this implies
that all individuals have the same type almost surely. The crucial step in the proof of part $(a)$ of the proposition is the following Lemma.

\begin{lemma}
\label{lem:char_gibbs}
Let $\lambda$ be a extremal $\mu-Gibbs$ measure. Then there exist $p\in[0,1]$ such that for all $v=(i_v,k_v)\in V_N$
$$
\lim_{m\rightarrow \infty}\lambda(X_v=a|\sigma_{-m})=p\quad \lambda- a.s.
$$
\end{lemma}

{\bf Proof.} \noemi{For fixed $v,$ the existence of the limit follows from the backward martingale convergence theorem, see \cite{JacodProtter}, page 233, and the fact that it is constant follows from the tail triviality of extremal Gibbs measures. It remains to prove that it is independent of $v.$ For this we couple the ancestral lines of two individuals $v_1$ and $v_2$ as in \cite{HammondSheffield} in as far as their $i-$coordinate (the generations) is concerned, and concernig the $k-$coordinate, that is, the label of the individual among the $N$ individuals per generation, we simply couple them completely, which does not change the law of the process. Hence the proof of \cite{HammondSheffield} goes through with only minor changes.}\hfill $\Box$\\

\noemi{For the rest of the proof of Proposition \ref{thm:char_gibbs_meas}, see \cite{HammondSheffield}. The main idea is as follows: For any finite set of individuals, there exists a (random) time $T$
before which the ancestral lines don't meet. This time is finite a.s., and in view of Lemma \ref{lem:char_gibbs},
there exists $p\in[0,1]$ such that the ancestors
alive just after time $T$ get their types independently with probability between $p-\varepsilon$ and $p+\varepsilon.$ This then
implies that $\lambda=\lambda_N^p,$ which, as we recall, conditional on $G\in {\mathcal T}_N$ is induced by the product Bernoulli 
measure on the components of $G$ with success parameter $p.$}

\end{appendix}

\noemi{\paragraph{Acknowledgments} The authors wish to thank an anonymous referee for making valuable suggestions which improved the presentation of the results considerably.}

\end{document}